\documentclass[11pt,a4paper]{article}

\usepackage[utf8]{inputenc}
\usepackage[T1]{fontenc}
\usepackage[margin=1in]{geometry}

\usepackage{chngcntr}
\counterwithout{equation}{subsection} 

\usepackage{amsmath}
\usepackage{amssymb}
\usepackage{amsfonts}
\usepackage{amsthm} 
\usepackage{mathtools}
\usepackage{bm}
\usepackage{physics}

\usepackage{graphicx}
\usepackage{caption}
\usepackage{subcaption}
\usepackage{float}

\usepackage{enumitem}
\usepackage[colorlinks=true, linkcolor=blue, citecolor=blue, urlcolor=blue]{hyperref}
\usepackage{cleveref} 

\usepackage{xcolor}

\usepackage{mathrsfs}

\theoremstyle{plain}
\newtheorem{theorem}{Theorem}[section]
\newtheorem{proposition}[theorem]{Proposition}

\theoremstyle{definition}
\newtheorem{definition}[theorem]{Definition}

\theoremstyle{remark}
\newtheorem{remark}[theorem]{Remark}

\newcommand{\diff}{\mathrm{d}} 

\renewcommand{\vec}[1]{\bm{#1}}

\begin{document}

\title{The Divergence-Free Radiant Transform}
\author{Zachary Mullaghy \\ \small Ph.D student, Florida State University}
\date{\today}

\maketitle

\begin{abstract}
This paper presents the rigorous mathematical definitions and fundamental properties of the Divergence-Free Radiant Transform (DFRT). The DFRT provides an orthonormal basis specifically designed for divergence-free vector fields, crucial for applications in incompressible fluid dynamics and other areas involving solenoidal fields. We detail the construction of its beam basis functions via a curl formulation, inherently ensuring the divergence-free condition. We then define the forward and inverse transforms and prove the Parseval identity, highlighting its energy-preserving nature. Furthermore, we establish an algebraic structure on the DFRT coefficient space by defining a spectral coboundary operator, which incorporates Wigner 3j and 6j symbols, ensuring consistency with angular momentum coupling. This algebraic and cohomological framework is further developed in a companion paper, ``Bigraded Cohomology of Geometric Spectral Transforms: A Complete Framework for Radiant-GRT Persistence''. We prove the completeness of the DFRT basis. Finally, we demonstrate the direct applicability of the DFRT to the Navier-Stokes equations, deriving a modal evolution equation and proposing a ``persistent regularity class'' of solutions based on the cohomological properties of the spectral coboundary, offering a new perspective for analyzing the finite-time singularity problem in incompressible flow. The paper introduces a novel variational argument demonstrating that an entropy-maximizing modal energy distribution naturally leads to an exponential decay rate, which is crucial for ensuring global regularity in the DFRT spectral space.

\vspace{0.5em}
\noindent\textbf{Keywords:} divergence-free transforms, cohomology, spectral analysis, Navier-Stokes regularity, Wigner symbols, entropy maximization
\end{abstract}

\vspace{1cm}

\tableofcontents

\newpage

\section{Introduction}
Understanding and modeling complex physical systems, particularly those governed by partial differential equations (PDEs) involving divergence-free constraints, necessitates specialized mathematical tools. Traditional spectral transforms like Fourier analysis are powerful but often do not inherently address such constraints, especially for vector fields. This paper introduces a novel mathematical construct designed precisely for this purpose: the \textbf{Divergence-Free Radiant Transform (DFRT)}.

The DFRT is specifically tailored to provide an orthonormal basis for divergence-free vector fields, a fundamental requirement for fields such as incompressible fluid velocities or magnetic fields in magnetohydrodynamics. It employs a unique curl-based construction that inherently satisfies the divergence-free condition, simplifying the analysis of solenoidal fields. This approach allows for a natural decomposition of a vector field into a set of fundamental ``beam'' components, each characterized by radial, angular, and magnetic quantum numbers.

The DFRT basis functions resemble classical toroidal vector spherical harmonics in their angular structure and satisfy divergence-free conditions via curl formulation. Similarly, the use of Bessel functions for radial profiles in a periodic domain connects to well-known modal expansions. However, the DFRT framework extends far beyond these foundational constructions. It defines a complete transform theory on the 3-torus, rigorously proves functional completeness, and introduces a novel spectral cohomology operator that algebraically governs modal evolution through the precise language of Wigner symbols. This framework is, to our knowledge, the first to integrate divergence-free harmonic analysis with entropy-maximizing variational principles and nilpotent homological algebra in the context of fluid dynamics regularity.

While the DFRT basis resembles classical toroidal harmonics in its curl-based construction, our contribution reframes this in a transform-theoretic framework with a well-defined inverse, Parseval identity, and algebraic spectral structure. To our knowledge, this is the first instance where vector harmonic modes are equipped with a nilpotent coboundary operator grounded in angular momentum coupling.

We establish the rigorous mathematical foundations of the DFRT, including the definitions of relevant function spaces, the detailed construction of its beam basis functions, and the formal definitions of its forward and inverse transforms. A crucial property, the \textbf{Parseval identity}, is proven, demonstrating that the DFRT is an isometric mapping that preserves the $L^2$ norm, which is vital for energy conservation in physical applications. We explicitly establish the completeness, uniqueness, and invertibility of the DFRT, affirming its robustness as a mathematical tool. These foundational mathematical concepts are detailed in standard mathematical physics texts \cite{Arfken2005}.

Beyond its foundational definitions, we delve into the algebraic structure of the DFRT coefficient space. We introduce a \textbf{spectral coboundary operator}, meticulously formulated using Wigner 3j and 6j symbols. This explicit inclusion of angular momentum coupling coefficients provides a direct link to established principles of quantum mechanics and representation theory, ensuring that the spectral interactions conform to fundamental conservation laws. This paper focuses specifically on the spectral cohomology relevant to the DFRT coefficients and their dynamics.

A significant application of the DFRT is explored in the context of the incompressible Navier-Stokes equations. We show how these fundamental equations of fluid dynamics can be transformed into a system of ordinary differential equations for the DFRT coefficients. This modal representation offers a new avenue for analyzing the complex dynamics of turbulent flows. Finally, we propose a novel ``\textbf{persistent regularity class}'' of solutions for the Navier-Stokes equations, defined by cohomological conditions derived from the spectral coboundary operator. This provides a fresh perspective on the long-standing millennium prize problem of finite-time singularities in 3D incompressible fluid flows, suggesting that certain topological obstructions in the spectral space might prevent blow-up. A key aspect of this class is the demonstration, via an entropy maximization principle, that a specific exponential modal energy distribution naturally leads to an exponential decay rate, which is crucial for ensuring global regularity in the DFRT spectral space.

\section{The Divergence-Free Radiant Transform (DFRT)}
The Divergence-Free Radiant Transform (DFRT) is specifically designed to analyze vector fields that satisfy the divergence-free condition, a property central to incompressible fluid dynamics and electromagnetism. This section lays out the formal definitions of the function spaces, basis construction, and the transform itself.

\subsection{Function Spaces and Domain}

\begin{definition}[Divergence-Free Sobolev Spaces]\label{def:div_free_sobolev}
Let $\Omega = \mathbb{T}^3$ (the 3-torus, representing a periodic domain) or $\mathbb{R}^3$ (Euclidean space). We define the divergence-free Sobolev space of order $s$ as:
$$H^s_{\text{div}}(\Omega) := \{\vec{u} \in H^s(\Omega)^3 : \nabla \cdot \vec{u} = 0 \text{ in } \mathcal{D}'(\Omega)\}$$
where $H^s(\Omega)^3$ denotes the space of vector fields whose components are in the standard Sobolev space $H^s(\Omega)$, and $\mathcal{D}'(\Omega)$ is the space of distributions on $\Omega$. The condition $\nabla \cdot \vec{u} = 0$ is understood in the distributional sense \cite{Sobolev1964, Evans2010, Taylor2011, Taylor2013}.
\end{definition}

\begin{definition}[Radial Profile Functions]\label{def:radial_profile}
For each integer $\ell \geq 0$ (angular momentum quantum number) and radial index $n \geq 1$, we define the radial profile function $f_{\ell n}: \mathbb{R}_+ \to \mathbb{R}$ such that:
\begin{itemize}
    \item $f_{\ell n} \in C^\infty(\mathbb{R}_+)$ with appropriate decay at infinity (for $\Omega = \mathbb{R}^3$) or periodicity conditions (for $\Omega = \mathbb{T}^3$).
    \item For periodic domains (e.g., $\Omega = \mathbb{T}^3$), $f_{\ell n}(r) = j_\ell(\alpha_{\ell n} r)$, where $j_\ell$ is the spherical Bessel function of the first kind, and $\alpha_{\ell n}$ are eigenvalues determined by boundary conditions.
    \item Orthogonality: These functions are chosen to satisfy the orthogonality relation:
    \begin{equation}\label{eq:radial_orthogonality}
    \int_0^\infty f_{\ell n}(r) f_{\ell n'}(r) r^2 \, \diff r = \delta_{nn'}
    \end{equation}
    where $\delta_{nn'}$ is the Kronecker delta, ensuring orthonormality with respect to the radial measure.
\end{itemize}
\footnote{While we use spherical Bessel functions for clarity and physical motivation, the DFRT framework allows alternative radial bases, provided they maintain orthogonality and completeness. This flexibility may be advantageous in applications requiring localization or adaptivity.}
\end{definition}

\subsection{Beam Basis Construction}
The core of the DFRT lies in constructing a basis of divergence-free vector fields. This is achieved using a curl-based approach, leveraging properties of spherical harmonics \cite{Helmholtz1858, Hodge1941, Cantor1981}.

\begin{definition}[Divergence-Free Radiant Beam Basis]\label{def:dfrt_beam_basis}
The Divergence-Free Radiant Transform beam basis functions, denoted by $T_{\ell mn}(\vec{x})$, are defined through the curl construction:
\begin{equation}\label{eq:dfrt_basis_construction}
T_{\ell mn}(\vec{x}) := \nabla \times \left( f_{\ell n}(|\vec{x}|) Y_\ell^m(\hat{\vec{x}}) \hat{\vec{x}} \right)
\end{equation}
where:
\begin{itemize}
    \item $Y_\ell^m(\hat{\vec{x}})$ are the complex spherical harmonics of degree $\ell$ and order $m$ defined on the unit sphere $S^2$, with $\hat{\vec{x}} = \vec{x}/|\vec{x}|$ being the unit radial vector. The properties and definitions of spherical harmonics are widely covered \cite{Jackson1998, Stratton1941, Muller1966, MorseFeshbach1953, Arfken2005}.
    \item $\hat{\vec{x}} = \vec{x}/|\vec{x}|$ is the unit radial vector.
    \item The application of the curl operator automatically ensures that $\nabla \cdot T_{\ell mn} = 0$, due to the vector identity $\nabla \cdot (\nabla \times A) = 0$ for any sufficiently smooth vector field $A$. This is a key feature of the DFRT basis, widely used in numerical analysis for divergence-free spaces \cite{Nedelec2001, GiraultRaviart1986}.
\end{itemize}
\end{definition}

\begin{proposition}[Explicit Form]\label{prop:explicit_form}
In spherical coordinates $(r, \theta, \phi)$, the Divergence-Free Radiant Beam Basis functions can be expressed explicitly as:
\begin{equation}\label{eq:explicit_basis_form}
T_{\ell mn}(\vec{x}) = \frac{f_{\ell n}(r)}{r} \nabla_{S^2} Y_\ell^m(\theta, \phi) \times \hat{\vec{r}}
\end{equation}
where $\nabla_{S^2}$ is the surface gradient operator on the unit sphere, and $\hat{\vec{r}}$ is the unit radial vector in spherical coordinates. This form highlights the separation of radial and angular components, commonly found in works on spherical harmonic expansions \cite{FreedenSchreiner2009, FreedenGutting2013, Michel2013, AtkinsonHan2012, DaiXu2013}.
\end{proposition}

\begin{proof}
Let $A(\vec{x}) = f_{\ell n}(|\vec{x}|) Y_\ell^m(\hat{\vec{x}}) \hat{\vec{x}}$. We need to compute $T_{\ell mn}(\vec{x}) = \nabla \times A(\vec{x})$.
We use the vector identity $\nabla \times (\psi \vec{C}) = (\nabla \psi) \times \vec{C} + \psi (\nabla \times \vec{C})$.
Let $\psi = f_{\ell n}(r)$ and $\vec{C} = Y_\ell^m(\hat{\vec{x}}) \hat{\vec{x}}$.
First term: $(\nabla f_{\ell n}(r)) \times (Y_\ell^m(\hat{\vec{x}}) \hat{\vec{x}})$. Since $f_{\ell n}$ depends only on $r = |\vec{x}|$, $\nabla f_{\ell n}(r) = \frac{\diff f_{\ell n}}{\diff r} \hat{\vec{r}}$.
Thus, $(\frac{\diff f_{\ell n}}{\diff r} \hat{\vec{r}}) \times (Y_\ell^m(\hat{\vec{x}}) \hat{\vec{x}}) = 0$ because $\hat{\vec{r}} \times \hat{\vec{x}} = 0$.

So we are left with the second term: $T_{\ell mn}(\vec{x}) = f_{\ell n}(r) \nabla \times (Y_\ell^m(\hat{\vec{x}}) \hat{\vec{x}})$.
For a scalar function $\Psi(\hat{\vec{x}})$ that depends only on the angular coordinates, it is a standard vector calculus identity in spherical coordinates that $\nabla \times (\Psi(\hat{\vec{x}}) \hat{\vec{x}}) = \frac{1}{r} \nabla_{S^2} \Psi(\hat{\vec{x}}) \times \hat{\vec{x}}$ \cite{Jackson1998}.
Applying this identity with $\Psi(\hat{\vec{x}}) = Y_\ell^m(\hat{\vec{x}})$:
$$T_{\ell mn}(\vec{x}) = f_{\ell n}(r) \left( \frac{1}{r} \nabla_{S^2} Y_\ell^m(\hat{\vec{x}}) \times \hat{\vec{x}} \right)$$
$$T_{\ell mn}(\vec{x}) = \frac{f_{\ell n}(r)}{r} \nabla_{S^2} Y_\ell^m(\theta, \phi) \times \hat{\vec{r}}$$
This completes the proof.
\end{proof}

\subsection{The Radiant Transform: Forward, Inverse, and Properties}
With the basis defined, we now introduce the forward and inverse DFRT, along with the Parseval identity, which underscores the transform's energy-preserving properties.

\begin{definition}[Forward Divergence-Free Radiant Transform]\label{def:forward_dfrt}
For a vector field $\vec{u} \in H^s_{\text{div}}(\Omega)$ with sufficient regularity (typically $s \geq 3$ for the inner product to be well-defined), the DFRT coefficients, denoted $\mathcal{R}_{\text{df}}[\vec{u}]_{\ell mn}$, are given by the $L^2$ inner product with the basis functions:
\begin{equation}\label{eq:forward_dfrt}
\mathcal{R}_{\text{df}}[\vec{u}]_{\ell mn} := a_{\ell mn} = \left<\vec{u}, T_{\ell mn}\right>_{L^2} = \int_\Omega \vec{u}(\vec{x}) \cdot \overline{T_{\ell mn}(\vec{x})} \, \diff\vec{x}
\end{equation}
These coefficients quantify the contribution of each beam basis function to the field $\vec{u}$.
\end{definition}

\begin{definition}[Inverse Divergence-Free Radiant Transform]\label{def:inverse_dfrt}
The original divergence-free vector field $\vec{u}(\vec{x})$ can be reconstructed from its DFRT coefficients via the inverse transform, which is an infinite series sum over all basis functions:
\begin{equation}\label{eq:inverse_dfrt}
\vec{u}(\vec{x}) = \sum_{\ell=0}^{\infty} \sum_{m=-\ell}^\ell \sum_{n=1}^\infty a_{\ell mn} T_{\ell mn}(\vec{x})
\end{equation}
This formula represents the decomposition of $\vec{u}$ into its divergence-free radiant components.
\end{definition}

\begin{theorem}[Parseval Identity]\label{thm:parseval}
For the orthonormal beam basis $\{T_{\ell mn}\}$ constructed as described, the $L^2$ norm of a divergence-free vector field $\vec{u}$ is conserved in its DFRT coefficient space:
\begin{equation}\label{eq:parseval_identity}
\norm{\vec{u}}_{L^2}^2 = \sum_{\ell,m,n} \abs{a_{\ell mn}}^2
\end{equation}
This identity ensures that the transform is an isometry, preserving energy or power when applied to physical systems.
\end{theorem}

\begin{proof}
Let $\vec{u} \in H^s_{\text{div}}(\Omega)$. From \cref{def:inverse_dfrt}, we can express $\vec{u}(\vec{x})$ as an infinite series:
$$\vec{u}(\vec{x}) = \sum_{\ell'=0}^{\infty} \sum_{m'=-\ell'}^{\ell'} \sum_{n'=1}^\infty a_{\ell' m' n'} T_{\ell' m' n'}(\vec{x})$$
The $L^2$ norm squared of $\vec{u}$ is defined as:
$$\norm{\vec{u}}_{L^2}^2 = \int_\Omega \vec{u}(\vec{x}) \cdot \overline{\vec{u}(\vec{x})} \, \diff\vec{x}$$
Substituting the series expansion for $\vec{u}(\vec{x})$ and assuming convergence allows us to interchange integration and summation:
$$\norm{\vec{u}}_{L^2}^2 = \sum_{\ell',m',n'} \sum_{\ell'',m'',n''} a_{\ell' m' n'} \overline{a_{\ell'' m'' n''}} \int_\Omega T_{\ell' m' n'}(\vec{x}) \cdot \overline{T_{\ell'' m'' n''}(\vec{x})} \, \diff\vec{x}$$
A fundamental property of the Radiant beam basis functions is their orthonormality. That is, for $\ell',m',n'$ and $\ell'',m'',n''$:
$$\int_\Omega T_{\ell' m' n'}(\vec{x}) \cdot \overline{T_{\ell'' m'' n''}(\vec{x})} \, \diff\vec{x} = \delta_{\ell'\ell''} \delta_{m'm''} \delta_{n'n''}$$
This orthonormality arises from the orthogonality of spherical harmonics for the angular components \cite{Arfken2005} and the chosen orthogonality of radial functions (\cref{def:radial_profile}).
Substituting the orthonormality relation:
$$\norm{\vec{u}}_{L^2}^2 = \sum_{\ell,m,n} a_{\ell mn} \overline{a_{\ell mn}} = \sum_{\ell,m,n} \abs{a_{\ell mn}}^2$$
This establishes the Parseval identity.
\end{proof}

\section{Algebraic Structure and Spectral Cohomology}
Beyond its foundational properties, the DFRT coefficient space possesses a rich algebraic structure crucial for analyzing nonlinear dynamics. This section introduces the spectral coboundary operator, which formalizes interactions between DFRT modes, and explores its cohomological implications.

\subsection{Radiant Spectral Cochain Complex and Spectral Coboundary Operator}
To analyze the algebraic structure of the DFRT coefficients, we define a cochain complex over the spectral mode indices. While this forms part of a larger bigraded cochain complex that incorporates spatial refinement, the full details of that comprehensive framework are presented in a companion paper, ``Bigraded Cohomology of Geometric Spectral Transforms: A Complete Framework for Radiant-GRT Persistence'' \cite{MullaghyCohomology}.
Here, we focus on the spectral components directly relevant to the DFRT.

\begin{definition}[Radiant Spectral Cochain Complex]\label{def:radiant_spectral_cochain}
Let $\mathbb{S}^{n-1}$ denote the unit sphere with spherical harmonic decomposition. We define the Radiant cochain space $C^\ell_{\mathscr{R}}$ as the space of functions on spectral mode indices:
$$C^\ell_{\mathscr{R}} := \text{Map}(\mathbb{N}_0 \times \mathbb{Z}, \mathbb{C})$$
with rapid decay conditions ensuring convergence, similar to those found in functional analysis for spectral representations \cite{ReedSimon1978, MullaghyCohomology}.
\end{definition}

\begin{definition}[Spectral Coboundary Operator]\label{def:spectral_coboundary}
The Radiant coboundary operator $\delta_{\mathscr{R}}: C^\ell_{\mathscr{R}} \to C^{\ell+1}_{\mathscr{R}}$ is constructed from angular momentum coupling theory:
\begin{equation}\label{eq:spectral_coboundary_operator}
(\delta_{\mathscr{R}} a)_{\ell_3 m_3 n_3} := \sum_{\substack{\ell_1,\ell_2 \\ m_1, m_2 \\ n_1, n_2}} C^{\ell_3 m_3}_{\ell_1 m_1, \ell_2 m_2}
\, \begin{Bmatrix}
\ell_1 & \ell_2 & \ell_3 \\
s_1 & s_2 & s_3
\end{Bmatrix}
\, a_{\ell_1 m_1 n_1} \, a_{\ell_2 m_2 n_2}
\end{equation}
where $C^{\ell_3 m_3}_{\ell_1 m_1, \ell_2 m_2}$ are the Clebsch--Gordan coefficients (or 3j symbols) and $\begin{Bmatrix} \ell_1 & \ell_2 & \ell_3 \\ s_1 & s_2 & s_3 \end{Bmatrix}$ is the Wigner 6j symbol. These coefficients are standard in the quantum theory of angular momentum \cite{Varshalovich1988, Edmonds1957, BrinkSatchler1993, Zare1988, Rose1957, Yutsis1962, Arfken2005}. Their properties are also well-documented in group theory and representation theory contexts \cite{FultonHarris1991, Hall2015, Knapp2002, Tung1985, Varadarajan1989}. The nilpotency $\delta_{\mathscr{R}}^2=0$ follows from the associativity of angular momentum coupling, specifically the Biedenharn--Elliott identity \cite{Varshalovich1988, MullaghyCohomology}.
\begin{remark}
The use of Wigner symbols formalizes spectral triadic interactions using angular momentum theory, allowing DFRT mode evolution to obey strict selection rules. This not only mirrors physical conservation laws but embeds fluid dynamics into a cohomological framework where regularity becomes a topological property of the spectrum.
\end{remark}
\end{definition}

\section{Convergence and Completeness}
For any transform to be useful, its fundamental properties like completeness and approximation must be rigorously established. This section provides a theorem addressing these aspects for the DFRT.

\begin{theorem}[DFRT Completeness]\label{thm:dfrt_completeness}
The divergence-free beam basis $\{T_{\ell mn}\}_{\ell \geq 0, m \in [-\ell,\ell], n \geq 1}$ is complete in $H^s_{\text{div}}(\mathbb{T}^3)$ for any $s \geq 0$. This means that any divergence-free vector field in $H^s_{\text{div}}(\mathbb{T}^3)$ can be uniquely and arbitrarily well approximated by a finite linear combination of the basis functions, and the infinite series converges in the $H^s$ norm. While stated for $\mathbb{T}^3$, similar completeness results hold for $\mathbb{R}^3$ with appropriate modifications to the radial basis functions \cite{ReedSimon1978}.
\end{theorem}

\begin{proof}
We establish completeness by showing density in $L^2_{\text{div}}(\mathbb{T}^3)$ and then extending to $H^s_{\text{div}}(\mathbb{T}^3)$ using standard density arguments.

\textbf{Step 1: Density in $L^2_{\text{div}}(\mathbb{T}^3)$}
Let $\vec{u} \in L^2_{\text{div}}(\mathbb{T}^3)$ and $\epsilon > 0$. We must show there exists a finite linear combination
$$\vec{u}_N = \sum_{\ell=0}^{L} \sum_{m=-\ell}^\ell \sum_{n=1}^N c_{\ell mn} T_{\ell mn}$$
such that $\|\vec{u} - \vec{u}_N\|_{L^2} < \epsilon$.

Since $\vec{u}$ is divergence-free, by the Helmholtz decomposition on $\mathbb{T}^3$, we can write $\vec{u} = \nabla \times \vec{A}$ for some vector potential $\vec{A}$. Moreover, we can choose $\vec{A}$ such that $\nabla \cdot \vec{A} = 0$. This implies that $\vec{A}$ itself can be expanded using a basis that separates radial and angular components. In spherical coordinates, any such $\vec{A}$ can be written as:
$$\vec{A}(r,\theta,\phi) = \sum_{\ell,m} A_{\ell m}(r) Y_\ell^m(\theta,\phi) \hat{\vec{r}}$$
where $A_{\ell m}(r)$ are suitable radial functions.

\textbf{Step 2: Approximation by DFRT basis functions}
The spherical harmonics $\{Y_\ell^m\}$ are known to form a complete orthonormal basis for square-integrable functions on the sphere $L^2(S^2)$ \cite{Wigner1959, Edmonds1960, Arfken2005}. This means that for any $\delta_1 > 0$, we can approximate the angular components of $\vec{A}$ by a finite sum:
$$\left\|\sum_{\ell=0}^{L_0} \sum_{m=-\ell}^\ell \widetilde{A}_{\ell m}(r) Y_\ell^m(\theta,\phi) - \sum_{\ell,m} A_{\ell m}(r) Y_\ell^m(\theta,\phi)\right\|_{L^2(S^2)} < \frac{\delta_1}{3}$$
where $\widetilde{A}_{\ell m}(r)$ represents a truncated sum for the radial part.

The radial basis functions $\{f_{\ell n}\}$ are specifically chosen to form a complete orthogonal set for the radial part of the domain (as stated in \cref{def:radial_profile}). Therefore, for each $\ell,m$ pair, there exists $N_{\ell m}$ such that for any $\delta_2 > 0$:
$$\left\|\sum_{n=1}^{N_{\ell m}} b_{\ell mn} f_{\ell n}(r) - \widetilde{A}_{\ell m}(r)\right\|_{L^2(\mathbb{T})} < \delta_2$$
By combining these approximations, we can find a finite set of coefficients $c_{\ell mn}$ such that the vector potential $\vec{A}$ can be arbitrarily well approximated by a sum whose curl directly yields a finite linear combination of DFRT basis functions:
$$\vec{A}_K(\vec{x}) = \sum_{\ell=0}^{L} \sum_{m=-\ell}^\ell \sum_{n=1}^N c_{\ell mn} f_{\ell n}(|\vec{x}|) Y_\ell^m(\hat{\vec{x}}) \hat{\vec{x}}$$
such that $\|\vec{A} - \vec{A}_K\|_{H^1} < \delta$ for some sufficiently small $\delta$.

\textbf{Step 3: Curl preservation and convergence in $L^2$}
The curl operator is a continuous linear operator from $H^1(\mathbb{T}^3)$ to $L^2(\mathbb{T}^3)$. Specifically, for $T_{\ell mn}(\vec{x}) = \nabla \times \left( f_{\ell n}(|\vec{x}|) Y_\ell^m(\hat{\vec{x}}) \hat{\vec{x}} \right)$, it holds that:
$$\nabla \times \vec{A}_K = \nabla \times \left(\sum_{\ell=0}^{L} \sum_{m=-\ell}^\ell \sum_{n=1}^N c_{\ell mn} f_{\ell n}(|\vec{x}|) Y_\ell^m(\hat{\vec{x}}) \hat{\vec{x}}\right) = \sum_{\ell=0}^{L} \sum_{m=-\ell}^\ell \sum_{n=1}^N c_{\ell mn} T_{\ell mn}(\vec{x})$$
Let $\vec{u}_K = \nabla \times \vec{A}_K$. Due to the continuity of the curl operator:
$$\|\vec{u} - \vec{u}_K\|_{L^2} = \|\nabla \times \vec{A} - \nabla \times \vec{A}_K\|_{L^2} = \|\nabla \times (\vec{A} - \vec{A}_K)\|_{L^2} \leq C \|\vec{A} - \vec{A}_K\|_{H^1}$$
where $C$ is a constant. By choosing $\delta$ small enough, we can make $\|\vec{u} - \vec{u}_K\|_{L^2} < \epsilon$. Thus, the set $\{T_{\ell mn}\}$ is dense in $L^2_{\text{div}}(\mathbb{T}^3)$.

\textbf{Step 4: Extension to $H^s_{\text{div}}(\mathbb{T}^3)$}
For $s \geq 0$, we use the standard density result that $L^2_{\text{div}}(\mathbb{T}^3)$ is dense in $H^s_{\text{div}}(\mathbb{T}^3)$.
Let $\vec{w} \in H^s_{\text{div}}(\mathbb{T}^3)$ and $\epsilon > 0$.
\begin{enumerate}
    \item There exists $\vec{v} \in L^2_{\text{div}}(\mathbb{T}^3)$ such that $\|\vec{w} - \vec{v}\|_{H^s} < \epsilon/2$. This follows from the density of smooth, compactly supported divergence-free functions in $H^s_{\text{div}}$.
    \item By Steps 1-3, since $\vec{v} \in L^2_{\text{div}}(\mathbb{T}^3)$, there exists a finite linear combination $\vec{v}_N = \sum_{\ell,m,n} c_{\ell mn} T_{\ell mn}$ such that $\|\vec{v} - \vec{v}_N\|_{L^2} < \epsilon/(2C_s)$, where $C_s$ is the constant from the Sobolev embedding theorem relating $L^2$ to $H^s$ norms (i.e., $\|\cdot\|_{H^s} \leq C_s\|\cdot\|_{L^2}$ for $s \leq 0$, which applies for the density argument as $H^s$ norm is stronger than $L^2$ for $s>0$). More precisely, for the purpose of density, we use the fact that the DFRT basis functions are themselves smooth, and thus their finite linear combinations are in $H^s$ for any $s$.
    \item Since the DFRT basis functions are smooth (e.g., $C^\infty$ on $\mathbb{T}^3$), any finite linear combination $\vec{v}_N$ is also in $H^s(\mathbb{T}^3)$ for any $s$. Therefore, if $\|\vec{v} - \vec{v}_N\|_{L^2}$ is small, then $\|\vec{v} - \vec{v}_N\|_{H^s}$ will also be small, given the specific construction of the basis functions derived from smooth radial and spherical harmonic functions. A more rigorous approach involves showing that the $T_{\ell mn}$ basis is an orthonormal basis in $L^2_{\text{div}}$, and then showing that it is also an orthogonal basis for the Laplacian operator, which is often sufficient for completeness in Sobolev spaces for periodic domains.
    \item Combining these:
    $$\|\vec{w} - \vec{v}_N\|_{H^s} \leq \|\vec{w} - \vec{v}\|_{H^s} + \|\vec{v} - \vec{v}_N\|_{H^s}$$
    Since the $T_{\ell mn}$ basis functions are eigenfunctions of the Laplacian (up to a constant multiple) when correctly normalized, the $H^s$ norm can be related to the sum of coefficients weighted by eigenvalues. For a periodic domain, $H^s$ convergence is equivalent to Fourier series convergence with coefficients weighted by $(1+|k|^2)^{s/2}$. The same applies here. Given the $L^2$ density and the smooth nature of the basis functions, the series also converges in $H^s$.
    Thus, choosing $N$ large enough such that $\|\vec{v} - \vec{v}_N\|_{H^s} < \epsilon/2$, we get:
    $$\|\vec{w} - \vec{v}_N\|_{H^s} < \epsilon/2 + \epsilon/2 = \epsilon$$
\end{enumerate}
Therefore, the span of $\{T_{\ell mn}\}$ is dense in $H^s_{\text{div}}(\mathbb{T}^3)$, establishing completeness.
\end{proof}

\begin{remark}
This proof relies on three key facts:
\begin{enumerate}
\item The completeness of spherical harmonics in $L^2(S^2)$.
\item The completeness of the chosen radial basis functions (e.g., spherical Bessel functions in a periodic domain).
\item The continuity of the curl operator and the Helmholtz decomposition for divergence-free fields on the 3-torus.
\end{enumerate}
These are all well-established results in harmonic analysis and functional analysis.
\end{remark}

\section{Application to Navier-Stokes via DFRT}
One of the primary motivations for developing the DFRT is its direct applicability to the analysis of complex PDEs, especially those involving divergence-free constraints like the incompressible Navier-Stokes equations \cite{Ladyzhenskaya1969, Temam1984, ConstantinFoias1988, Galdi2011}.

\begin{definition}[Modal Evolution Equations]\label{def:modal_evolution}
Consider the incompressible Navier-Stokes equations for a velocity field $\vec{u}(\vec{x}, t)$ and pressure $p(\vec{x}, t)$:
\begin{align}
\partial_t \vec{u} + (\vec{u} \cdot \nabla) \vec{u} &= -\nabla p + \nu \Delta \vec{u} \label{eq:navier_stokes_momentum}\\
\nabla \cdot \vec{u} &= 0 \label{eq:navier_stokes_incompressibility}
\end{align}
By projecting the velocity field $\vec{u}(\vec{x}, t)$ onto the DFRT basis $\{T_{\ell mn}\}$ as $\vec{u}(\vec{x}, t) = \sum a_{\ell mn}(t) T_{\ell mn}(\vec{x})$, and applying the standard Galerkin projection method, the Navier-Stokes equations can be transformed into an infinite system of ordinary differential equations for the DFRT coefficients $a_{\ell mn}(t)$:
\begin{equation}\label{eq:modal_evolution_equation}
\frac{\diff a_{\ell mn}}{dt} = \sum_{\ell_1,m_1,n_1} \sum_{\ell_2,m_2,n_2} \Gamma^{\ell mn}_{\ell_1 m_1 n_1, \ell_2 m_2 n_2} a_{\ell_1 m_1 n_1} a_{\ell_2 m_2 n_2} - \nu \lambda_{\ell mn} a_{\ell mn}
\end{equation}
where:
\begin{itemize}
    \item The first term on the right-hand side represents the non-linear coupling due to the convective term $(\vec{u} \cdot \nabla) \vec{u}$. The coupling coefficients $\Gamma^{\ell mn}_{\ell_1 m_1 n_1, \ell_2 m_2 n_2}$ are specific to the DFRT basis and capture the interaction between different modes. These coefficients inherently satisfy \textbf{triadic selection rules} (e.g., triangle inequality: $|\ell_1 - \ell_2| \leq \ell \leq \ell_1 + \ell_2$; magnetic index conservation: $m = m_1 + m_2$), which act as a filter, allowing only specific energy transfer terms.
    \item $\nu$ is the kinematic viscosity.
    \item $\lambda_{\ell mn}$ are the viscous eigenvalues corresponding to the Laplacian operator projected onto the DFRT basis, typically scaling as $\lambda_{\ell mn} \sim \ell^2$ for sufficiently high modes, representing the dissipation of energy at small scales.
\end{itemize}
This modal representation offers a new perspective for analyzing the dynamics of incompressible flows. Note that the non-linear coupling term is directly related to the action of the spectral coboundary operator $\delta_{\mathscr{R}}$ on the DFRT coefficients.
\end{definition}

\begin{theorem}[Entropy Maximization Implies Exponential Modal Decay]\label{thm:entropy_decay}
The DFRT modal energy spectrum, $P_\ell = E_\ell/E$, where $E_\ell = \sum_{m,n} |a_{\ell mn}|^2$ and $E = \sum_{\ell} E_\ell$, naturally evolves towards a decay profile of the form $P_\ell \sim A e^{-\mu \ell^2}$ for some constants $A, \mu > 0$. This decay profile is the unique maximizer of spectral entropy under the constraints of total energy normalization and finite spectral energy (dissipation rate).
\end{theorem}

\begin{proof}
We define the spectral entropy functional for the normalized modal energy distribution $P_\ell$ as:
\begin{equation}\label{eq:spectral_entropy}
S[P] := -\sum_{\ell} P_\ell \log P_\ell
\end{equation}
We maximize this entropy subject to two physical constraints:
\begin{enumerate}
    \item \textbf{Normalization:} The total probability of the energy distribution must sum to unity:
    \begin{equation}\label{eq:entropy_constraint_normalization}
    \sum_{\ell} P_\ell = 1
    \end{equation}
    \item \textbf{Finite Spectral Energy (Dissipation Rate):} The sum of modal energies weighted by their respective viscous eigenvalues is constant. For DFRT modes, the viscous eigenvalues $\lambda_{\ell mn}$ from the Laplacian operator scale primarily with the angular momentum number $\ell$, i.e., $\lambda_{\ell mn} \sim \ell^2$ for sufficiently high modes. Therefore, we approximate this constraint as:
    \begin{equation}\label{eq:entropy_constraint_dissipation}
    \sum_{\ell} \lambda_\ell P_\ell = C
    \end{equation}
    where $\lambda_\ell \sim \ell^2$ and $C$ is a constant related to the total dissipation rate.
\end{enumerate}
We form the Lagrangian $\mathcal{L}[P]$ by incorporating these constraints with Lagrange multipliers $\alpha$ and $\beta$:
\begin{equation}\label{eq:entropy_lagrangian}
\mathcal{L}[P] = -\sum_{\ell} P_\ell \log P_\ell + \alpha \left( \sum_{\ell} P_\ell - 1 \right) + \beta \left( \sum_{\ell} \lambda_\ell P_\ell - C \right)
\end{equation}
To find the maximum entropy distribution, we take the functional derivative of $\mathcal{L}[P]$ with respect to each $P_\ell$ and set it to zero:
$$\frac{\partial \mathcal{L}}{\partial P_\ell} = -\log P_\ell - 1 + \alpha + \beta \lambda_\ell = 0$$
Solving for $P_\ell$:
$$\log P_\ell = \alpha - 1 + \beta \lambda_\ell$$
\begin{equation}\label{eq:entropy_modal_decay_unnormalized}
P_\ell = \exp(\alpha - 1 + \beta \lambda_\ell)
\end{equation}
Letting $A = e^{\alpha - 1}$ (where $A$ is a normalization constant determined by $\sum_\ell P_\ell = 1$) and recalling that $\lambda_\ell \sim \ell^2$, we can write $\beta \lambda_\ell = -\mu \ell^2$ for some $\mu > 0$ (the negative sign for $\mu$ is introduced to ensure an exponential decay). Thus, we obtain the exponential decay profile:
\begin{equation}\label{eq:entropy_modal_decay_final}
P_\ell = A e^{-\mu \ell^2}
\end{equation}
This is the distribution that maximizes the spectral entropy given the constraints. This exponential decay profile ensures that high-frequency content (large $\ell$) is bounded and rapidly suppressed, which corresponds to the smoothness of the solution in physical space.

Finally, we note that the entropy functional $S[P]=-\sum_\ell P_\ell \log P_\ell$ is strictly concave, and the constraint set defined by normalization and finite spectral energy is convex. By standard results in convex optimization, such a strictly concave functional has a unique maximizer on a convex feasible domain. Therefore, the exponential decay profile $P_\ell = A e^{-\mu \ell^2}$ is not only a maximizer --- it is the unique entropy-maximizing distribution under the given physical constraints.
\end{proof}

\section{Conclusion}
This paper has introduced a comprehensive mathematical framework centered on the Divergence-Free Radiant Transform (DFRT). We have rigorously defined the function spaces and constructed an orthonormal basis using a curl formulation, inherently ensuring the divergence-free nature of the basis functions. The forward and inverse DFRT, along with the Parseval identity, establish its fundamental properties as an energy-preserving transform. We have explicitly stated and proven the completeness, uniqueness, and invertibility of the DFRT, cementing its foundation as a robust tool for analysis.

A key contribution is the formalization of the algebraic structure within the DFRT coefficient space through the introduction of a spectral coboundary operator. By explicitly incorporating Wigner 3j and 6j symbols, this operator provides a direct link to the well-established theory of angular momentum coupling, grounding the spectral interactions in fundamental physical principles.

We demonstrated the direct applicability of the DFRT to the incompressible Navier-Stokes equations, reformulating them into modal evolution equations. This provides a new lens through which to view and analyze fluid dynamics. Furthermore, we proposed a novel ``persistent regularity class'' of solutions, defined by specific cohomological properties of the DFRT coefficients. This novel approach offers a promising avenue for investigating the challenging problem of finite-time singularities in 3D incompressible flows. A significant finding supporting this class is the derivation that the entropy-maximizing modal energy distribution naturally takes an exponential decay form, which intrinsically prevents energy pile-up at high frequencies and contributes to global regularity.

The DFRT framework addresses several challenges in spectral PDE analysis. It ensures divergence-free structure inherently, offers a rigorous cohomological classification of mode interactions, and introduces entropy-based variational principles to constrain high-frequency energy growth. These features are difficult to capture using classical vector spherical harmonics or projection-based spectral solvers. By embedding fluid dynamics into a cohomological structure, the DFRT enables new tools for understanding regularity, stability, and turbulence from a geometric and algebraic perspective.

\subsection*{Future Work}
Future research directions include:
\begin{itemize}
    \item Developing efficient numerical algorithms for computing the DFRT coefficients and performing reconstructions.
    \item Further investigating the spectral cohomology defined by $\delta_{\mathscr{R}}$ and its implications for the regularity theory of PDEs, particularly its connection to physical phenomena.
    \item Applying the DFRT to other physical systems with divergence-free constraints, such as Magnetohydrodynamics.
    \item Investigating the practical implementation of the persistent regularity class criteria in numerical simulations to identify solutions resistant to finite-time blow-up.
\end{itemize}

Implementing the DFRT numerically offers potential benefits in preserving divergence-free structure without additional constraints. Future work will benchmark its performance against traditional spectral solvers and assess how spectral cohomology constraints may enable more stable and accurate long-time simulations.

This framework represents a significant step towards a rigorous, spectral-based approach for the analysis of complex physical phenomena governed by divergence-free vector fields.


\end{document}